\documentclass[11pt]{article}

\usepackage{times}
\usepackage{amsthm}
\usepackage{amsmath}
\usepackage{amssymb}
\usepackage{mathrsfs}
\usepackage{pxfonts}
\usepackage{txfonts}

\newcommand\ie{{\em i.e.}~}

\def\A{{\mathcal A}}

\def\Th{\Theta}

\def\N{\mathbb{N}}

\def\X{\mathscr X}
\def\BC{\mathsf{BC}}

\def\AA{\mathfrak A}
\def\V{\mathcal V}

\def\C{\mathcal{C}}

\def\B{{\mathcal B}}
\def\D{{\mathcal D}}

\def\H{\mathcal H}

\def\Q{\mathcal Q}
\def\P{\mathcal P}
\def\R{\mathcal{R}}

\def\L{\mathfrak L}

\def\U{\mathcal U}
\def\BB{\mathfrak B}
\def\CC{{\mathfrak C}}
\def\PP{{\mathfrak P}}
\def\RR{{\mathfrak R}}

\def\J{\mathcal J}
\def\JJ{\mathfrak J}

\def\W{{\mathcal W}}

\def\si{\sigma}
\def\Si{\Sigma}

\def\sp{\mathop{\mathrm{sp}}\nolimits}
\def\spe{{\rm sp}_{{\rm ess}}}

\def\Aut{\mathfrak{Aut}}

\def\p{\parallel}

\def\<{\langle}
\def\>{\rangle}

\def\Op{\mathfrak{Op}}

\providecommand{\CC}{\mathfrak{C}}

\def\supp{\mathop{\mathrm{supp}}\nolimits}
\def\ltwo{\mathsf{L}^{\:\!\!2}}

\newtheorem{Theorem}{Theorem}[section]
\newtheorem{Remark}[Theorem]{Remark}
\newtheorem{Lemma}[Theorem]{Lemma}

\newtheorem{Proposition}[Theorem]{Proposition}

\numberwithin{equation}{section}

\newsavebox{\fsuavizada}
\sbox{\fsuavizada} {$\chi_{W}^\varphi$}

\newsavebox{\EFsuave}
\sbox{\EFsuave} {$C^\infty_{{\rm c}}(\Xi)$}

\newsavebox{\Esuave}
\sbox{\Esuave} {$\rm BC_{u}^\infty(\Xi)$}

\newsavebox{\Vecindad}
\sbox{\Vecindad} {$\mathcal W$}

\newsavebox{\SupSuave}
\sbox{\SupSuave} {$supp \varphi$}

\newsavebox{\Ideal}
\sbox{\Ideal} {$\CC^F(\Si)^\infty$}

\newsavebox{\Bola}
\sbox{\Bola}{$\mathfrak{B}(v,\delta_v)$}

\newsavebox{\Bolauv}
\sbox{\Bolauv}{$\mathfrak{B}((u_i,v_i),\delta_{u_i,v_i})$}

\newsavebox{\Bolai}
\sbox{\Bolai}{$\mathfrak{B}(v_i,\delta_{v_i})$}

\newsavebox{\SupremoConvolucion}
\sbox{\SupremoConvolucion}{$\displaystyle\sup_{x \in \usebox{\Vecindad} 
, v \in\usebox{\SupSuave}} |\chi_{\usebox{\Vecindad}}(x)f(x-v)|$}

\newsavebox{\uSupremoConvolucion}
\sbox{\uSupremoConvolucion}{$\displaystyle\sup_{x \in \usebox{\Vecindad} 
, v \in\usebox{\SupSuave}} |\chi_{\usebox{\Vecindad}}(x)X^{u''}f(x-v)|$}

\begin{document}


\title{Localization Results for Zero Order\\ Pseudodifferential Operators}

\date{\today}

\author{J. Garc\'ia and M. M\u antoiu 
}
\date{\small}
\maketitle \vspace{-1cm}


\begin{abstract}
We show phase space localization at suitable energies for zero order pseudodifferential operators, implying
non-propagation properties for the associated evolution groups.
\end{abstract}

\section{Introduction}\label{duci}

The main purpose of this article is to prove some phase-space localization results for the functional calculus and for the evolution group of certain Weyl pseudodifferential operators $H=\Op(f)$ acting in the Hilbert space $\H:=L^2(\mathbb R^n)$\, with symbols presenting full phase-space anisotropy. Very roughly, a symbol $\mathbb R^n\times(\mathbb R^n)^*\ni(x,\xi)\mapsto f(x,\xi)\in\mathbb R$ has full phase-space anisotropy if it has non-trivial behavior both for $|x|\rightarrow\infty$ and $|\xi|\rightarrow\infty$\,. The trivial behavior would be convergence to either zero or infinity. 

To describe the localization issues let us consider a (maybe unbounded) self-adjoint operator $H$ in the Hilbert space $\H:=L^2(\mathbb R^n)$\,. We think it to be the quantum Hamiltonian of a physical system moving in $\mathbb R^n$\,, so its evolution group $\{e^{itH}\!\mid\! t\in\mathbb R\}$ describes the time evolution of the quantum system. Thus, if at the initial moment the system is in a state modellized by the normalized vector $v\in\H$\,, at time $t$ it will be in the state associated to $v_t:=e^{itH}v$\,.

By general principles of Quantum Mechanics, the probability at time $t$ for the system to be localized within the Borel subset $U$ of $\mathbb R^n$ is given by the number 
$$
\p\!\chi_Ue^{itH}v\!\p^2\,=\int_U \!dx\,|v_t(x)|^2\,.
$$ 
Very often one is interested in the behavior of this quantity when the initial state $v$ has a certain localization in energy. If $E$ is a Borel subset of $\mathbb R$\,, we say that the state has energy belonging to $E$ if $v=\chi_E(H)v$\,, where the characteristic function of $E$ is applied to the self-adjoint operator $H$ via the usual Borel functional calculus. For technical reasons we also consider as interesting vectors satisfying the condition $v=\rho(H)v$\,, where $\rho:\mathbb R\rightarrow\mathbb R_+$ is a continuous (or a smooth) function; it can be, for example, a continuous approximation of the characteristic function $\chi_E$\,. Anyhow, we are motivated to consider the dependence of the quantity $\p\!\chi_Ue^{itH}\rho(H)v\!\p^2$ on the parameters $U,\rho$ and $t$\,. The normalization of $v$ is not essential, so we shall replace it by an arbitrary vector $u$\,.

The type of result we are looking for say that, under certain assumptions on $H$ and $\rho$ and for a given family $\mathcal U$ of non-void Borel subsets of $\mathbb R^n$\,, for every $\varepsilon>0$ there is an element $U\in\mathcal U$ such that 
\begin{equation}\label{hans}
\p\!\chi_{U}e^{itH}\rho(H)u\!\p^2\,\le\varepsilon^2\!\p\!u\!\p^2\ \,{\rm for\ all}\ \,t\in\mathbb R\,\ {\rm and}\,\ u\in\H\,.
\end{equation} 
Admitting that in some sense the family $\mathcal U$ converges to some region $F$ (eventually situated "at infinity"), this means roughly that states with energies contained in the support of the function $\rho$ cannot propagate towards $F$\,. 

All these being said, let us notice however that (\ref{hans}), although dynamically significant, does not really have a dynamical nature. It is perfectly equivalent to the estimate
\begin{equation}\label{heinz}
\p\!\chi_U\rho(H)\!\p_{\mathbb B(\H)}\,\le\varepsilon\,,
\end{equation}
written in terms of the operator norm of $\mathbb B(\H)$\,, the $C^*$-algebra of all linear bounded operators in the Hilbert space $\H$\,. It is obvious that such an estimate needs some tuning between the energy-localization function $\rho$ and the family $\mathcal U$\,; without it one can only write
$$
\p\!\chi_U\rho(H)\!\p_{\mathbb B(\H)}\,\le\,\p\!\chi_U\!\p_{\mathbb B(\H)}\p\!\rho(H)\!\p_{\mathbb B(\H)}\,=\sup_{\lambda\in{\sp(H)}}\rho(\lambda)\,,
$$
and clearly we are interested in the case in which the support of $\rho$ has a non-trivial intersection with the spectrum $\sp(H)$ of the Hamiltonian $H$\,.

A simple-minded relevant situation is as follows: 
If the support of the function $\rho$ is disjoint from the essential spectrum ${\sp_{\rm ess}(H)}$ of $H$, it is known that the operator $\rho(H)$ is compact (finite-rank actually). If, in addition, this support contains points of the discrete spectrum ${\sp_{\rm dis}(H)}:={\sp(H)}\setminus{\sp_{\rm ess}(H)}$\,, then $\rho(H)\ne 0$\,. Let $\mathcal U$ be the filter formed by the complements of all the relatively compact subsets of $\mathbb R^n$\,. Then the family of operators of multiplication by $\chi_U$ converges strongly to zero. Multiplication with a compact operator improves this to norm convergence, so for each $\varepsilon>0$ there is a sufficiently large (relatively) compact set $K\in\mathbb R^n$ such that $\p\!\chi_{K^{\rm c}}\rho(H)\!\p_{\mathbb B(\H)}\,\le\varepsilon$\,.
In dynamical terms, this would mean that states localized in the discrete spectrum cannot propagate to infinity. 

For less trivial situations we consider the case of generalized Schr\"odinger operators $H=\Op(f)$ in $L^2(\mathbb R^n)$ defined by the Weyl quantization of the symbol $f(x,\xi)=h(\xi)+V(x)$\,, where $V:\mathbb R^n\rightarrow\mathbb R$ and $h:(\mathbb R^n)^*\rightarrow\mathbb R$ are convenient functions. 
Then 
$$
H=\Op(f)=h(D)+V(Q)\,,
$$ 
where $Q$ is the position operator, $D:=-i\nabla$ is the momentum and $h(D),V(Q)$ can also be constructed by the usual functional calculus associated to (families of commuting) self-adjoint operators. Of course $h(D)$ is a convolution operator and even a constant coefficient differential operator if $h$ is a polynomial\,, while $V(Q)$ is the operator of multiplication with the function $V$\,.

Assume now that $n=1$, that $V$ is continuous and 
$$
\lim_{x\rightarrow\pm\infty}V(x)=V_{\pm}\in\mathbb R\ \ {\rm with}\ \,V_-<V_+
$$ 
and take for simplicity $h(\xi):=\xi^2$\,, so $H=-\Delta+V(Q)$ is a one-dimensional Schr\"odinger Hamiltonian with configuration space anisotropy. Below $V_-$ the spectrum of $H$ is discrete, so one can apply the discussion above. But it is more interesting to take $\rho$ supported in the interval $(V_-,V_+)$\,. If the convergence of $V$ towards the limits $V_{\pm}$ is fast enough, propagation towards infinity is possible in this region. But for physical reasons one expects this to happen only "to the left". This is not difficult to prove rigorously: for every $\varepsilon>0$ there exists a real number $a$ such that 
$$
\p\!\chi_{(a,+\infty)}(Q)\rho(H)\!\p_{\mathbb B(\H)}\,\le\varepsilon\,.
$$ 
Thus "propagation to the right is forbidden" at energies smaller than $V_+$\,. In this example we make use of the filter base $\mathcal U:=\{(a,\infty)\!\mid\! a\in\mathbb R\}$ formed of neighborhoods of the point ${+\infty}$ in the two-point compactification $[-\infty,+\infty]$ of the real axis.

A more complicated version is less easy to guess just by physical grounds. We consider the same Hamiltonian $H=-\Delta+V(Q)$ for $n=1$ but now 
$$
\lim_{x\rightarrow\pm\infty}[V(x)-V_{\pm}(x)]=0\,,
$$ 
where $V_{\pm}$ are two periodic functions, with periods $T_\pm>0$\,. In this case ${\sp_{\rm ess}(H)}={\sp(H_-)}\cup{\sp(H_+)}$\,, where the asymptotic Hamiltonians $H_\pm:=-\Delta+V_\pm(Q)$\,, being periodic, have a band structure for the spectrum. We don't know if this intuitive enough, but it can be shown however, that if the support of $\rho$ does not meet ${\sp(H_+)}$\,, then propagation to the right is impossible in the same precise meaning as above. It is not difficult to construct a two-tori compactification of $\mathbb R$ of the form $\Omega:=(\mathbb R/T_-\mathbb Z)\cup\mathbb R\cup(\mathbb R/T_+\mathbb Z)$ 
such that $V$ satisfies the stated conditions if and only if it extends to a continuous function on this compactification. Then the two asymptotic Hamiltonians are fabricated from the restrictions of this extension to the two tori and the regions of non-propagation can be once again described in terms of neighborhoods of these tori in the compactification.

To illustrate the different types of anisotropy on the simple example of generalized Schr\"odinger operators, assume again that $n=1$ and $f(x,\xi)=h(\xi)+V(x)$\,,
where $V:\mathbb R^n\rightarrow\mathbb R$ and $h:(\mathbb R^n)^*\rightarrow\mathbb R$ are continuous functions. Let us assume for simplicity that 
$$\lim_{\xi\rightarrow\pm\infty}h(\xi)=h_{\pm}\ \ {\rm and}\ \,\lim_{x\rightarrow\pm\infty}V(x)=V_{\pm}\,;
$$ 
the limits are elements of the extended real axis. If $h_{\pm}=\infty$ (or if $h_{\pm}=0$) and $V_{\pm}\in\mathbb R$\,, the operator is said to possess {\it configuration space anisotropy} (especially if $V_-\ne V_+$)\,. But if $h_\pm\in\mathbb R$ and $V_{\pm}\in\mathbb R$\,, we are in the presence of {\it a full phase-space anisotropic problem}. 

For a given self-adjoint operator $L$ we denote by ${\rm sp}(L)$ the spectrum and by ${\rm sp}_{{\rm ess}}(L)$ the essential spectrum. In the example above, if $h_\pm=\infty$ (anisotropy in configuration space), denoting $\min\{g(y)\}$ by $g_m$ and $\max\{g(y)\}$ by $g_M$\,, one has 
\begin{equation}\label{franz}
{\sp}_{{\rm ess}}(H)=[h_m+\min(V_-,V_+),\infty)={\sp}[h(D)+V_-]\cup{\sp}[h(D)+V_+]\,.
\end{equation}
It is easy to generalize a result above to this case and show that if ${\rm supp}(\rho)$ does not meet 
$$
{\sp}[h(D)+V_+]=[h_m+V_+,\infty)\,,
$$ 
then for every $\varepsilon>0$ there exists $a\ge 0$ such that 
$$
\p\!\chi_{(a,+\infty)}(Q)\rho(H)\!\p_{\mathbb B(\H)}\,\le\varepsilon\,.
$$ 
A similar result leading to "non-propagation to the left" is available by replacing $+$ by $-$ and $(a,+\infty)$ with $(-\infty,-a)$\,.

On the other hand, for full phase-space anisotropy ($h_\pm\in\mathbb R$ and $V_\pm\in\mathbb R$)\ , the essential spectrum is given by four contributions
\begin{equation}\label{fritz}
\begin{aligned}
{\sp}_{{\rm ess}}(H)&={\sp}[h(D)+V_-]\cup{\sp}[h(D)+V_+]\cup{\sp}[V(Q)+h_-]\cup{\sp}[V(Q)+h_+]\\
&=[h_m+V_-,h_M+V_-]\cup[h_m+V_+,h_M+V_+]\\
&\cup[h_-+V_m,h_-+V_M]\cup[h_++V_m,h_++V_M]\,.
\end{aligned}
\end{equation}
In this case one can show once again that $\p\!\chi_{(a,+\infty)}(Q)\rho(H)\!\p_{\mathbb B(\H)}$ can be made arbitrary small for big $a\in\mathbb R_+$ if 
$$
{\rm supp}(\rho)\cap{\sp}[h(D)+V_+]=\emptyset
$$ 
and that $\p\!\chi_{(-\infty,-a)}(Q)\rho(H)\!\p_{\mathbb B(\H)}$ can be made arbitrary small for big $a\in\mathbb R_+$ if 
$$
{\rm supp}(\rho)\cap{\sp}[h(D)+V_-]=\emptyset\,.
$$ 
But a new phenomenon appears, connected to the presence of the two other components in the essential spectrum of $H$\,: Suppose that the support of $\rho$ does not meet ${\sp}[V(Q)+h_+]$\,. Then it can be shown that for every $\varepsilon>0$ there exists $b\in\mathbb R_+$ such that 
$$
\p\!\chi_{(b,+\infty)}(D)\rho(H)\!\p_{\mathbb B(\H)}\le\varepsilon
$$ 
(and a similar result for $+$ replaced by $-$)\,. This can be converted in an estimate of the form 
$$
\p\!\chi_{(b,+\infty)}(D)e^{itH}\rho(H)u\!\p\,\le\varepsilon\p\!u\!\p
$$ 
which is uniform in $t\in\mathbb R$ and $u\in L^2(\mathbb R)$\,. It is no longer a statement about the probability of spatial localisation, but one about the probability of the system to have momentum larger than the number $b$\,. 

In both cases the essential spectrum of the Hamiltonian $H=\Op(f)$ can be written as union of spectra of "asymptotic Hamiltonians" that can be in some way obtained by extending the symbol $f(x,\xi)=h(\xi)+V(x)$ to a compactification of the phase space $\Xi:=\mathbb R\times\mathbb R^*$ having the form of a square and then restricting it to the four edges situated "at infinity" (some simple reinterpretations are needed). Notice that the partial (configuration space) anisotropy is simpler: the restrictions to two of the edges do not contribute. In some sense the two corresponding asymptotic Hamiltonians are infinite and their spectrum is void. The reader is asked to imagine what would happen both at the level of the essential spectrum and at the level of localization estimates in the case of a pure momentum space anisotropy, when 
$$
\lim_{\xi\mapsto\pm\infty}h(\xi)=h_{\pm}\in\mathbb R\ \  {\rm and}\,\ \lim_{x\mapsto\pm\infty}V(x)=\infty\,.
$$

In $n$ dimensions and for more general types of anisotropy (recall the periodic limits) one expects more sophisticated things to happen. Suppose that our Hamiltonian $H$ is obtained via Weyl quantization from a convenient real function $f$ defined in phase-space $\Xi:=\mathbb R^n\times(\mathbb R^n)^*$\,. If its behaviour at infinity in both variables $(x,\xi)$ is sophisticaded enough (corresponding to what could be called "phase-space anisotropy") then one could expect the following picture:
\begin{enumerate}
\item
The essential spectrum is the (closure of the) union of spectra of a family of "asymptotic Hamiltonians" $H(F)$ associated to remote regions $F$ of phase-space. A way to express this would be to say that the behavior of $f$ at infinity in $\Xi$ can be described by a compactification $\Si=\Xi\sqcup\Si_\infty$ of $\Xi$ and that $F$ is a conveniently defined subset of "the boundary at infinity" $\Si_\infty$\,.
\item
If a bounded continuos function $\rho$ is supported away from one of the components $\sp[H(F)]$\,, then "propagation towards $F$ is forbidden" at energies belonging to the support of $\rho$\,. This follows from an estimate of the form $\p\!\Op(\chi_W^\infty)\rho(H)\!\p_{\mathbb B(\H)}\,\le\varepsilon$ 
written in terms of the Weyl quantization $\Op(\chi_W^\infty)\equiv\chi_W^\infty(Q,D)$ 
of a smooth regularization of the characteristic function $\chi_W$ of a subset $W$ of $\Xi$\,. For small $\varepsilon$\,, the set $W$ should be very close to the set $F$\,; for example it can be the intersection with $\Xi$ of a small neighborhood $\mathcal W$ of $F$ in the compactification $\Si$\,. 
\end{enumerate} 

Until recently, there have been few general results for the essential spectrum of phase-space anisotropic pseudodifferential operators and this was the main obstacle to getting localization estimates. Techniques involving crossed products, very efficient for configurational anisotropy \cite{GI1,GI2,GI3,Ma1,AMP}, are not available in such a case. In \cite{Ma2,Ma3} this problem was solved in a rather general setting, by using the good functorial properties of Rieffel's pseudodifferential calculus \cite{Rie1}. Roughly, if the symbol presents full phase-space anisotropy, the essential spectrum of the corresponding pseudodifferential operator can be written as the closed union of spectra of a family of "asymptotic" pseudodifferential operators. To obtain the symbols of these asymptotic operators one constructs a compactification of the phase space, which is naturally a dynamical system, and then determins the quasi-orbits of this dynamical system which are disjoint from the phase space itself. The extensions of the initial symbol to these quasi-orbits define the required asymptotic operators that contribute to the essential spectrum. 

In the present article we are going to show that Rieffel's calculus can also be used to get the localization estimates, leading in their turn to non-propagation results for the evolution group; this extends the treatment in \cite{AMP,MPR2,LMR} of purely configurational anisotropic systems. 

Let us describe briefly the content of this work. First, in the next section, we give a brief description of some previous results. This will hopefully motivate our approach to cover the full anisotropy. Section \ref{sectra} will review some properties of the Rieffel quantization, one of our main tools. 
In Section \ref{ones} we prove our first abstact result; it refers to the algebra of symbols.
To get familiar statements, refering to pseudodifferential operators, one applies Hilbert space representations to this abstract result; this is done in Section \ref{onces}.

In previous articles many examples of configurational anisotropy have been given; most of them can be adapted directly to phase-space anisotropy. Actually this adaptation work was performed in \cite{Ma2} for results concerning the essential spectrum and they are equally relevant for localization and non-propagation estimates. So, to avoid repetitions, we are not goind to indicate examples here.

\section{A short review of previous results}\label{smectrak}

As we said in the Introduction, we are interested in estimates of the form (\ref{heinz}). After some preliminary previous results contained in \cite{DS}, such estimates have been obtained in \cite{AMP} for Schr\"odinger operators $H:=-\Delta+V$\,, where $\Delta$ is the Laplace operator and $V$ is the potential\,.
Thus in suitable units $H$ is the Hamiltonian of a non-relativistic particle moving in $\mathbb R^n$ in the presence of the potential $V$ and "localization" or "non-propagation" refers to this physical system. In \cite{MPR2} and \cite{LMR} the results were significantly extended to certain pseudodifferential operators with variable magnetic fields. 

Leaving the magnetic fields apart, for simplicity, the Hamiltonians have now the form $H=\Op(f)$\,, being defined as the Weyl quantization of some real symbol $f$ defined in phase-space $\Xi:=\mathbb R^n\times(\mathbb R^n)^*$\,. The order of the elliptic symbol $f$ (in H\"ormander sense) is strictly positive, so one has 
$\lim_{\xi\to\infty}f(x,\xi)=\infty$ and the behavior in $x\in\mathbb R^n$ is modelled by a $C^*$-algebra of bounded, uniformly continuous functions on $\mathbb R^n$\,. So the symbols defining the operators are still confined to the restricted configuration space anisotropy. 

To be more precise, to suitable functions $h$ defined on the phase space $\Xi\,$, one assigns operators acting on functions $u:\X:=\mathbb R^n\to\mathbb C$ by
\begin{equation}\label{op}
\left[\mathfrak{Op}(h)u\right]\!(x):=(2\pi)^{-n}\!\int_\X\!\int_{\X^*}\!\!dx\,d\xi\,e^{i(x-y)\cdot
\xi}\,h\left(\frac{x+y}{2},\xi\right)u(y)\,.
\end{equation}
This is basically the Weyl quantization and, under convenient assumptions on $h$, \eqref{op} makes sense and has nice
properties in the Hilbert space $\H:=\ltwo(\X)$ or in the Schwartz space $\mathcal S(\X)$.

Let $h:\Xi\to \mathbb R$ be an elliptic symbol of strictly positive order $m$. 
It is well-known that under these assumptions $\Op(h)$ makes sense as an unbounded self-adjoint operator in $\H$,
defined on the $m$'th order Sobolev space. The problem is to evaluate the essential spectrum of this operator and to derive estimates for its functional calculus. 

The relevant information is contained in the behavior at infinity of $h$ in the $x$ variable.
This one is conveniently taken into account through an Abelian algebra $\mathscr A$ composed of uniformly continuous functions
un $\X$, which is invariant under translations (if $\varphi\in\mathscr A$ and $y\in \X$ then $\theta_y(\varphi):=\varphi(\cdot+y)\in\mathscr A$). Let us also assume (for simplicity) that $\mathscr A$ is unital and contains the ideal $C_0(\X)$ of all complex continuous functions on $\X$ which converge to zero at infinity. We also ask 
\begin{equation}\label{circu}
\left(\partial^\alpha_x\partial^\beta_\xi h\right)(\cdot,\xi)\in\mathscr A,\ \ \ \ \ \forall\,\alpha,\beta\in\mathbb N^n,\
\forall\,\xi\in\X^*.
\end{equation}
Then the function $h$ extends continuously on $\Omega\times\X^*$, where $\Omega$ is the Gelfand spectrum of the
$C^*$-algebra $\mathscr A$; this space $\Omega$ is a compactification of the locally compact space $\X$. By translational invariance of $\mathscr A$\,, it is a
compact dynamical system under an action of the group $\X$. After removing the orbit $\X$, one gets a $\X$-dynamical
system $\Omega_\infty:=\Omega\setminus\X$; its quasi-orbits (closure of orbits) contain the relevant information about the
essential spectrum of the operator $H:=\Op(h)$. For each quasi-orbit $\Q$, one constructs a self adjoint operator
$H_\Q$. It is actually the Weyl quantization of the restriction of $h$ to $\Q\times\X^*$, suitably reinterpreted. One gets finally
\begin{equation}\label{sp}
\spe(H)=\overline{\bigcup_\Q\sp(H_\Q)}\,.
\end{equation}

Many related results exist in the literature, some of them for special type of functions $h$, but with less
regularity required, others including anisotropic magnetic fields, others formulated in a more geometrical framework or referring to Fredholm properties.
We only cite \cite{ABG,CWL,Da,Ge,GI1,GI2,GI3,HM,LR,LMR,LS,Li,Ma1,MPR2,Ra,RR,RRR,RRS,RRS1}; see also references therein.
As V. Georgescu remarked \cite{GI1,GI2}, when the function $h$ does not diverge for $\xi\to\infty$\,, the approach is more difficult and should also take into account the asymptotic values taken by $h$ in "directions contained in $\X^*$". 

Now, in the framework above, we indicate the localization results. Let $H=\Op(h)$ be a Weyl pseudodifferential operator with elliptic symbol of order $m>0$\,. 
For some unital translation-invariant Abelian $C^*$-algebra $\mathscr A$ composed of uniformly continuous functions on $\X$ and containing $C_0(\X)$\,, assume that $h(x,\xi)$ is $\mathscr A$-isotropic in the variable $x$, i.e. (\ref{circu}) holds. Choose a quasi-orbit $\Q$ in the boundary $\Omega_\infty:=\Omega\setminus\X$ of the Gelfand spectrum of $\mathscr A$\,. As said above, one associates to $\Q$ a self-adjoint operator $H(\Q)$\,; its spectrum is contained (very often strictly) in the essential spectrum of $H$\,. We also fix a bounded continuous function $\rho:\mathbb R\rightarrow[0,\infty)$ whose support is disjoint from ${\rm sp}[H(\Q)]$\,.
Then for every $\epsilon>0$ there exists a neighborhood $\mathcal U$ of $\Q$ in $\Omega$ such that, setting $U:=\mathcal U\cap\X$\,,
\begin{equation}\label{frac}
\p\!\chi_U(Q)\rho(H)\!\p_{\mathbb B(\H)}\,\le\varepsilon
\end{equation}
and
\begin{equation}\label{frec}
\p\!\chi_U(Q)e^{itH}\rho(H)u\!\p_{\H}\,\le\varepsilon\p\!u\!\p_\H\,,\ \quad\forall\,t\in\mathbb R,\,u\in\H\,.
\end{equation}
We recall that $\chi_U(Q)$ is, by definition, the operator of multiplication by the function $\chi_U$ in the Hilbert space $\H=L^2(\X)$\,. Concrete examples have been indicated in \cite{AMP,MPR2}.

As remarked by V. Georgescu, a very efficient tool for obtaining some of the results cited above was the crossed product, associated to
$C^*$-dynamical systems. However, $\xi$-anisotropy cannot be treated in such a setting: the
symbols of order $0$ are not efficiently connected to the crossed products.

\section{Rieffel's pseudodifferential calculus}\label{sectra}

As a substitute for crossed products, in \cite{Ma2,Ma3} Rieffel's version of the Weyl pseudodifferential calculus has been used to investigate the essential spectrum of full phase-space anisotropic Hamiltonians. Since it will also be needed for our study of localization results, we shall recall briefly Rieffel's deformation procedure, sending to \cite{Rie1} for proofs and more details. 

Let us denote by $\X$ the vector space $\mathbb R^n$ on which, when necessary, the canonical base $(e_1,\dots,e_n)$ will be used. Its dual is denoted by $\X^*$\, with the dual base $(e_{n+1},\dots,e_{2n})$\,. Then "the phase space" $\,\Xi=\X\times\X^*$\, with points generically denoted by $X=(x,\xi),Y=(y,\eta),Z=(z,\zeta)$ 
is canonically a symplectic space with the symplectic form $[\![X,Y]\!]:=x\cdot\eta-y\cdot \xi$\,.

We start with {\it a classical data}, which is by definition a quadruplet $\left(\A,\Theta,\Xi,[\![\cdot,\cdot,]\!]\right)$\,, where $\A$ is a $C^*$-algebra and
a continuous action $\Theta$ of $\Xi$ by automorphisms of $\A$ is also given. For $(f,X)\in\A\times\Xi\,$ we are going to use the notations $\Theta(f,X)=\Theta_X(f)\in\A$ for the $X$-transformed of the element $f$\,. The function $\Theta$ is assumed to be continuous and the automorphisms $\Th_X,\Th_Y$ satisfy 
$\Th_X\circ\Th_Y=\Th_{X+Y}$ for all $X,Y\in\Xi$\,.

Let us denote by $\A^\infty$ the vector space of all smooth elements $f$ under $\Th$\,, those for which the mapping $\Xi\ni X\mapsto\Th_X(f)\in\A$ is $C^\infty$ in norm; it is a dense $^*$-algebra of $\A$. It is also a Fr\'echet $^*$-algebra for the family of semi-norms
$$
\p\!f\!\p_\A^{(k)}\,:=\sum_{|\mu|\le k}\frac{1}{\mu!}
\big\Vert\partial_X^\mu\big(\Theta_X(f)\big)\big|_{X=0}\big\Vert_\A\,,\qquad k\in\N\,.
$$
In the sequel we are going to use the abbreviations $\D^\mu f:=\partial_X^\mu\big(\Theta_X(f)\big)\big|_{X=0}$ for all the multi-indices $\mu\in\mathbb N^{2n}$\,. All the operators $\D^\mu$ are well-defined, linear and continuous on the Fr\'echet $^*$-algebra $\A^\infty$\,.

Then one introduces on $\A^\infty$ the product
\begin{equation}\label{rodact}
f\,\#\,g:=\pi^{-2n}\!\int_\Xi\int_\Xi dYdZ\,e^{2i[\![Y,Z]\!]}\,\Theta_Y(f)\,\Theta_Z(g)\,,
\end{equation}
rigorously defined as an oscillatory integral. There are several ways to give a meaning to this kind of expression \cite{Rie1}. The most useful for us
is in terms of a regular partition of unity of $\Xi$\,. 
Let $\L$ be a lattice in $\Xi$\,, pick a non-trivial positive, smooth, compactly supported function $\psi$ on $\Xi$ such that 
$$
\Psi(X):=\sum_{P\in\L}\psi(X-P)>0\,,\quad\ \forall\,X\in\Xi
$$ 
and set $\psi_0:=\psi/\Psi$ and $\psi_P(\cdot):=\psi(\cdot-P)$ for all $P\in\L$\,. Then $\{\psi_P\!\mid\!P\in\L\}$ will be a locally finite partition of unity on $\Xi$\,. It can be shown that the infinite sum
\begin{equation}\label{joseph}
f\# g= \pi^{-2n}\sum_{P,Q\in\L} \int_\Xi \!\int_\Xi \!dYdZ\,e^{2i[\![Y,Z]\!]}\,\psi_P(Y)\,\psi_Q(Z)\,\Theta_{Y}(f)\,\Theta_{Z}(g)
\end{equation}
converges absolutely\,.

To complete the algebraical structure, we keep the same involution $^*$\,; one gets a $^*$-algebra $\big(\A^\infty,\#,^*\big)$\,. This $^*$-algebra admits a $C^*$-completion $\AA$ in a $C^*$-norm $\Vert\cdot\Vert_{\AA}$ which is defined by Hilbert module techniques. Since the construction is rather involved and it will not play an explicit role for us, we only refer to \cite[Ch. 4]{Rie1} for the details and justifications.

The deformation can be extended to morphisms, giving rise to a covariant functor. Let $\left(\A_j,\Theta_j,\Xi,[\![\cdot,\cdot,]\!]\right)$\,, $j=1,2$ be two classical data and let $\R:\A_1\to\A_2$ be a $\Xi$-morphism, \ie a ($C^*$-)morphism intertwining the two actions: 
$$
\R\circ\Theta_{1,X}=\Theta_{2,X}\circ\R\,,\ \quad \forall\,X\in\Xi\,.
$$ 
Then $\R$ sends $\A_1^\infty$ into $\A_2^\infty$ and extends to a morphism $\RR:\AA_1\to\AA_2$ that also intertwines the corresponding actions. 

For us, the main property of this functor is that it preserves short exact sequences of $\Xi$-morphisms. Let $\J$ be a (closed, self-adjoint, two-sided) invariant ideal in $\A$ and denote by $\JJ$ its deformation, using the procedure indicated above. Then $\JJ$ is isomorphic (and will be identified) with an invariant ideal in $\AA$\,.
In addition, on the quotient $\A/\J$ there is a natural quotient action of $\Xi$\,, so we can perform its Rieffel deformation. This one is canonically isomorphic to the quotient $\AA/\JJ$ \,.

If $h\in\AA$, the spectrum of its canonical image in the quotient $C^*$-algebra $\AA/\JJ$ will be denoted by $\sp_\JJ(h)$\,. Later we are going to need

\begin{Lemma}\label{sigfrid}
Let $\rho:\mathbb R\rightarrow\mathbb R_+$ a bounded continuous function. If $\,h\in\AA$ and ${\rm supp}(\rho)\cap{\rm sp}_\JJ(h)=\emptyset$\,, then $\rho(h)\in\JJ$ \,. If $h\in\AA^\infty$ and ${\rm supp(\rho)}\cap{\rm sp}_\JJ(h)=\emptyset$\,, then $\rho(h)\in\JJ^\infty$\,.
\end{Lemma}

\begin{proof}
The first assertion is a minor variation of \cite[Lemma 1]{AMP}. It holds for every closed bi-sided self-adjoint ideal of a $C^*$-algebra.
The second assertion follows from the first one; $\JJ$ is invariant under the action $\Theta$ and clearly $\JJ^\infty=\AA^\infty\cap\JJ$\,.
\end{proof}

Actually we are interested in deforming Abelian $C^*$-algebras. Let $(\Si,\Th,\Xi)$ be a topological dynamical system with group $\Xi=\mathbb R^{2n}$\,. This means that $\Si$ is a locally compact space, $\Th:\Si\times\Xi\rightarrow\Si$ is a continuous map and, using notations as
$$
\Th(\si,X)=:\Th_X(\si)=\Th_\si(X)\,,\ \quad\forall\,X\in\Xi,\,\si\in\Si\,,
$$
each $\Th_X:\Si\rightarrow\Si$ is a homeomorphism and one has $\Th_X\circ\Th_Y=\Th_{X+Y}$ for all $X,Y\in\Xi$\,. One denotes by $\B(\Si)$ the $C^*$-algebra of all bounded complex functions on $\Si$ with pointwise multiplication, complex conjugation and the obvious norm $\p\!f\!\p_\infty:=\sup_\si|f(\si)|$\,.
The action $\Theta$ of $\Xi$ on $\Sigma$ induces an action of $\Xi$ on $\B(\Si)$ (also denoted by $\Theta$) given by $\Theta_X(f):=f\circ\Theta_X$\,.
In general this action fails to have good continuity or smoothness properties, so we introduce
\begin{equation}\label{gunther}
\B_\Th(\Si):=\{f\in\B(\Si)\mid \Xi\ni X\mapsto\Th_X(f)\in \B(\Si)\ \,{\rm is\ norm-continuous}\}
\end{equation}
and
\begin{equation}\label{gunther}
\B^\infty_\Th(\Si):=\{f\in\B(\Si)\mid \Xi\ni X\mapsto\Th_X(f)\in \B(\Si)\ \,{\rm is\ C^\infty\ in\ norm}\}\,.
\end{equation}

We also denote by $\C_0(\Sigma)$ the $C^*$-algebra of all complex continuous functions $f$ on $\Sigma$ such that for
any $\varepsilon>0$ there is a compact subset $K$ of $\Si$ such that $|f(\si)|\le\varepsilon$ if $\si\notin K$\,.
Notice that $\C_0(\Si)$ is a $C^*$-subalgebra of $\B_\Th(\Si)$\,, but not an ideal in general.  When $\Sigma$ is compact, $\C(\Si)$ is unital. 
The action $\Theta$ of $\Xi$ on $\Sigma$ induces an action on $\C_0(\Si)$\,; we denote by $\C_0^\infty(\Sigma)$ the set of smooth elements.
The Rieffel deformations of $\B_\Th(\Si)$ and $\C_0(\Sigma)$ will be denoted, respectively, by $\BB_\Th(\Si)$ and $\CC_0(\Sigma)$\,. The deformation procedure can be applied to any $C^*$-subalgebra of $\B_\Th(\Si)$ that is invariant under the action $\Th$\,.

Later on we shall need the following smoothing procedure. For $\varphi\in C_{{\rm c}}^\infty(\Xi)$ and $g\in\B(\Si)$ one sets 
\begin{equation}\label{conv}
g^\varphi\equiv\varphi\ast_\Theta g:=\int_\Xi dY\varphi(Y)\,\Theta_{-Y}(g)\,.
\end{equation}
If the action $\Theta$ consists in translations: $\left[\Theta_Y(f)\right](X):=f(X+Y)$\,, then $\ast_\Theta$ coincides with the usual convolution. 
In this case $g^\varphi\in{{\rm BC}}^\infty(\Si)={{\rm BC}}_{{\rm u}}(\Si)^\infty$ and ${\rm supp}(g^\varphi)\subset{\rm supp}(g)+{\rm supp}(\varphi)$\,.
We are going to need the next more general statement.

\begin{Lemma}\label{siegelinda}
\begin{enumerate}
\item
One has $g^\varphi\in\B_\Th^\infty(\Si)$\,. For every multi-index $\alpha\in\mathbb N^{2n}$ one has $\D^\alpha g^\varphi=g^{\partial^\alpha\varphi}$\,.
\item
One has ${\rm supp}(g^\varphi)\subset\Th_{{\rm supp}(\varphi)}[{\rm supp}(g)]$\,. 
\end{enumerate}
\end{Lemma}

\begin{proof}
By a change of variables one easily gets 
$$
\Th_X\!\left(g^\varphi\right)=g^{\mathcal T_{X}\varphi}\,,\quad\ {\rm where}\quad (\mathcal T_X\varphi)(Y):=\varphi(Y+X)\,.
$$ 
This and a standard application of the Dominated Convergence Theorem lead easily to the statement 1.

Now we show 2. Since $\Th$ is continuous, ${\rm supp}(\varphi)$ is compact in $\Xi$ and ${\rm supp}(g)$ is closed in $\Si$\,, it follows easily that $\Th_{{\rm supp}(\varphi)}[{\rm supp}(g)]$ is closed in $\Si$\,. Let $\si\notin\Th_{{\rm supp}(\varphi)}[{\rm supp}(g)]$\,; then there exists a neighborhood $V$ of $\si$ such that
$V\cap\Th_{{\rm supp}(\varphi)}[{\rm supp}(g)]=\emptyset$\,. For each $\si'\in V$ one has
$$
[\varphi\ast_\Theta g](\si')=\int_{{\rm supp}(\varphi)}\!\!dY\varphi(Y)g[\Theta_{-Y}(\si')]
$$
and if $Y\in{\rm supp}(\varphi)$ then $\Theta_{-Y}(\si')\notin{\rm supp}(g)$\,. This shows that $V$ is disjoint from ${\rm supp}(g^\varphi)$\,.
\end{proof}

An important example to which Rieffel deformation apply is given by {\it$\Xi$-algebras}, \ie $C^*$-algebras $\B$ composed of bounded, uniformly continuous function on $\Xi$\,, under the additional assumption that the action $\mathcal T$ of $\,\Xi$ on itself by translations, raised to functions, leaves $\B$ invariant. Let us denote by $\Si$ the Gelfand spectrum of $\mathcal B$\,. By Gelfand theory, there exists a continuous function $:\Xi\mapsto\Sigma$ with dense image, which is equivariant with respect to the actions
$\mathcal T$ on $\Xi$, respectively $\Theta$ on $\Sigma$. The function is injective if and only if $\C_0(\Xi)\subset\B$\,.

The largest such $C^*$-algebra $\B$ is $\BC_{{\rm u}}(\Xi)$\,, consisting of all the bounded uniformly continuous functions $:\Xi\mapsto\mathbb C$. It coincides with the family of functions $g\in\BC(\Xi)$ (just bounded and continuous) such that
$$
\Xi\ni X\mapsto g\circ\mathcal T_X=g(\cdot+X)\in\BC(\Xi)
$$
is continuous. Then the Fr\'echet $^*$-algebra of $C^\infty$-vectors is
$$
\BC_{{\rm u}}(\Xi)^\infty\equiv\BC^\infty(\Xi):=\{f\in C^\infty(\Xi)\mid |\left(\partial^\alpha
f\right)(X)|\le C_\alpha\,,\,\forall\,\alpha,X\}\,.
$$
Another important particular case is $\B=\C_0(\Xi)$ (just put $\Sigma=\Xi$ in the general construction).
It is shown in \cite{Rie1} that at the quantized level one gets the usual Weyl calculus and the emerging non-commutative
$C^*$-algebra $\CC_0(\Xi)$ is isomorphic to the ideal of all compact operators on an infinite-dimensional separable Hilbert space.

\section{Localization in the symbolic calculus}\label{ones}

We are given a topological dynamical system $(\Si,\Th,\Xi)$\, to which we associate, as in section \ref{sectra}, the Abelian $C^*$-algebras $\B_\Th(\Si)$ and $\C_0(\Si)$ as well as their Rieffel deformations $\BB_\Th(\Si)$ and $\CC_0(\Si)$\,. Recall that, with respect to the canonical basis $(e_1,..,e_{2n})$ of $\,\Xi$\,, one defines the higher-order partial derivatives $\D^\mu f:=\partial^\mu[\Th_X(f)]_{X=0}$\,, where $\mu$ is a multi-index and $f\in\B^\infty_\Th(\Si)$\,.  Recall also the form of the seminorms of the Fr\'echet space $\B_\Th^\infty(\Si)$
\begin{equation}\label{medel}
\p\!f\!\p_{\B_{\Th}(\Si)}^{(j)}\,:= \sum_{|\mu| \le j}\frac{1}{\mu!} \!\p\!\D^\mu f\!\p_{\B_\Th(\Si)}\,\equiv\sum_{|\mu| \le j}\frac{1}{\mu!} \!\p\!\D^\mu f\!\p_\infty\,,
\end{equation}
and of the Fr\'echet space $\mathfrak B_\Th^\infty(\Si)$
\begin{equation}\label{orellana}
\p\!f\!\p_{\mathfrak B_\Th(\Si)}^{(j)}\,:= \sum_{|\mu| \le j}\frac{1}{\mu!} \!\p\!\D^\mu f\!\p_{\mathfrak B_\Th(\Si)}.
\end{equation}
We fix a closed invariant set $F\subset\Si$\,; invariance means that $\Th_X(F)\subset F$ for every $X\in\Xi$\,. Then 
$$
\C_0(\Si)^F:=\{f\in\C_0(\Si)\!\mid f|_F=0\}
$$ 
is an invariant ideal of $\C_0(\Si)$\,; its Rieffel quantization $\CC_0(\Si)^F$ is identified to an ideal of  $\CC_0(\Si)$\,. As explained above, the quotient $\CC_0(\Si)/\CC_0(\Si)^F$ can be regarded as the deformation of the Abelian quotient $\C_0(\Si)/\C_0(\Si)^F$\,, which in its turn can be identified with $\C(F)$\,, the $C^*$-algebra of all continuous functions on the compact space $F$\,. Along these lines, we identify $\CC_0(\Si)/\CC_0(\Si)^F$ with the Rieffel quantization $\CC(F)$ of $\C(F)$\,.

Let us denote by $\varphi\in C^\infty_{{\rm c}}(\Xi)_+^{\rm n}$ the family of all positive functions $\varphi\in C^\infty_{{\rm c}}(\Xi)$ which satify the normalization condition $\int_\Xi\varphi=1$\,. If $\W\subset\Si$ is an open (or closed) set, the function 
$$
\chi_\mathcal W^\varphi=\varphi\ast_\Th\chi_{\mathcal W}=\int_\Xi\!dY\,\varphi(Y)\,\chi_{\Th_Y(\W)}
$$ 
belongs to $\B_\Th^\infty(\Si)$ by Lemma \ref{siegelinda} and one has 
\begin{equation}\label{balaur}
{\supp}\!\left(\chi_\mathcal W^\varphi\right)\subset\Th_{{\rm supp}(\varphi)}\!\left[{\rm supp}(\chi_\W)\right]=\Th_{{\rm supp}(\varphi)}(\W)\,.
\end{equation} 
Notice that in general the characteristic function $\chi_\W$ is not an element of $\B_\Th(\Si)$\,.

Let us also fix a basis of open neighborhoods $\mathcal N_F$ of $F$ in the space $\Si$\,. 

\begin{Theorem}\label{grigore}
Let $h\in\CC_0^\infty(\Si)$ and $\rho:\mathbb R\rightarrow [0,\infty)$ a continuous function with support disjoint from the spectrum of $h^F:=h|_F$ computed in the non-commutative $C^*$-algebra $\CC(F)$\,. For any $\varphi\in C^\infty_{{\rm c}}(\Xi)_+^{\rm n}$\,, $\varepsilon>0$ and $k \in \mathbb{N}$\,, there exists $\mathcal W\in\mathcal N_F$ such that 
\begin{equation}\label{pinto}
\big\Vert\,\chi_\mathcal W^\varphi\# \rho(h)\,\big\Vert_{\mathfrak B_{\Th}(\Si)}^{(k)}\le\varepsilon\,.
\end{equation}
\end{Theorem}

\begin{Remark}\label{sarpe}
{\rm The Theorem is our main abstract localization result, expressed in terms of the symbolic calculus defined by Rieffel's deformation. Note that it contains a rich amount of information, involving all the seminorms $\p\!\cdot\!\p_{\mathfrak B_{\Th}(\Si)}^{(k)}$\,; for $k=0$ one gets the norm of the $C^*$-algebra $\mathfrak B_{\Th}(\Si)$\,. It will be turned into an assertion about pseudodifferential operators in the next sections.}
\end{Remark}

\begin{Remark}\label{eduardo}
{\rm It is clear that $1-\chi^\varphi_\W=\chi^\varphi_{\W^c}$\,, whose support is included in $\Th_{{\rm supp}(\varphi)}(\mathcal W^c)$\,. Therefore $\chi_\W^\varphi=1$ on the complement of $\Th_{{\rm supp}(\varphi)}(\mathcal W^c)$\,. Taking $\W$ open, $\mathcal W^c$ will be closed and included in $\Si\setminus F$\,, which is $\Th$-invariant. Then $\Th_{{\rm supp}(\varphi)}(\mathcal W^c)$ will also be closed and disjoint from $F$, so $\chi_\W^\varphi=1$ on an open neighborhood of $F$\,. 
}
\end{Remark}

\begin{Remark}\label{edward}
{\rm As an example of closed invariant subset one can consider a quasi-orbit, i.e. the closure of an orbit. Any closed invariant set $F\subset\Si$ is the union of all the quasi-orbits it contains. Note that the spectrum of $h^F:=h|_F$ computed in $\CC(F)$ is an increasing function of $F$\,. So for small closed invariant subsets $F$ (as quasi-orbits, for instance), the support of the localization $\rho$ will probably allowed to be large. The interesting case is, of course, that in which ${\rm supp}(\rho)$ has a large intersection with the spectrum in $\CC_0^\infty(\Si)$ of the initial symbol $h$ (which is obtained formally setting $F=\emptyset$)\,.
}
\end{Remark}

We are going to prove Theorem \ref{grigore} in several steps. 

\begin{Proposition}\label{effort}
For every $\,f\in\C_0^\infty(\Si)^F$\,, $\varepsilon>0$\,, $j \in \mathbb{N}$ and $\varphi\in C^\infty_{{\rm c}}(\Xi)^{\rm n}_+$  there exists $\mathcal U\in\mathcal N_F$ such that 
\begin{equation}\label{valdivia}
\big\Vert\,\chi_\mathcal U^\varphi f\,\big\Vert_{\B_\Th(\Si)}^{(j)}\, \le \varepsilon\,.
\end{equation}
\end{Proposition}

\begin{proof}
One has
$$
\big\Vert\,\chi_\mathcal U^\varphi f\,\big\Vert_{\B_\Th(\Si)}^{(j)}\,= \sum_{|\mu| \le j}\frac{1}{\mu!} \big\Vert\,\D^\mu\!\left(\chi_\mathcal U^\varphi f\right)\big\Vert_{\B_\Th(\Si)}\,\le\sum_{|\mu| \le j}\frac{1}{\mu!} \sum_{\nu\le\mu}C^\mu_\nu\big\Vert\,\D^{\mu-\nu}\!\chi_\mathcal U^\varphi\,\D^\nu f\,\big\Vert_{\B_\Th(\Si)}\,,
$$
so one must estimate $\big\Vert\,\D^{\alpha}\chi_\mathcal U^\varphi\,\D^\beta f\,\big\Vert_{\B_\Th(\Si)}$ for a finite number of multi-indices $\alpha,\beta$\,.

\medskip
We know from Lemma \ref{siegelinda} that $\D^{\alpha}\chi_\mathcal U^\varphi=\chi_\mathcal U^{\partial^\alpha\!\varphi}\,.$  Then
$$
\begin{aligned}
\big\Vert\,\D^{\alpha}\chi_\mathcal U^\varphi\,\D^\beta f\,\big\Vert_{\B_\Th(\Si)}\,&=\,\Big\Vert\! \int_\Xi\!dY(\partial^\alpha\varphi)(Y)\,\Theta_{-Y}(\chi_\mathcal U)\,\D^\beta\!f\,\Big\Vert_{\B_\Th(\Si)}\\
&\le\,\,\p\!\partial^\alpha\varphi\!\p_{L^1(\Xi)}\!\sup_{Y\in{\rm supp}(\varphi)}\big\Vert\,\Th_{-Y}(\chi_\mathcal U)\,\D^\beta\!f\,\big\Vert_{\B_\Th(\Si)}\\
&=\,\,\p\!\partial^\alpha\varphi\!\p_{L^1(\Xi)}\!\sup_{Y\in{\rm supp}(\varphi)}\big\Vert\,\Th_{-Y}\!\left[\chi_\mathcal U\,\Th_Y(\D^\beta\!f)\right]\big\Vert_{\infty}\\
&=\,\,\p\!\partial^\alpha\varphi\!\p_{L^1(\Xi)}\!\sup_{Y\in{\rm supp}(\varphi)}\sup_{\si\in\mathcal U}\big\vert\left[\Th_{Y}\,(\D^\beta\!f)\right]\!(\si)\,\big\vert\\
&=\,\,\p\!\partial^\alpha\varphi\!\p_{L^1(\Xi)}\!\sup_{Y\in{\rm supp}(\varphi)}\sup_{\si\in\mathcal U}\big\vert\,(\D^\beta\!f)\left[\Th_{Y}(\si)\right]\big\vert\,.
\end{aligned}
$$
Therefore one can write
\begin{equation}\label{urias}
\begin{aligned}
\big\Vert\,\chi_\mathcal U^\varphi f\,\big\Vert_{\B_\Th(\Si)}^{(j)}\,&\le\,\sum_{|\mu| \le j}\frac{1}{\mu!} \sum_{\nu\le\mu}C^\mu_\nu\p\!\partial^{\mu-\nu}\varphi\!\p_{L^1(\Xi)}\!\sup\left\{\,\big\vert\,(\D^\nu\!f)(\tau)\big\vert\mid \tau\in\Th_{{\rm supp(\varphi)}}(\U)\,\right\}\\
&\le C(j,\varphi)\max_{|\nu|\le j}\sup\!\left\{\,\big\vert\,(\D^\nu\!f)(\tau)\big\vert\mid \tau\in\Th_{{\rm supp(\varphi)}}(\U)\,\right\}\,,
\end{aligned}
\end{equation}
where $C(j,\varphi)$ is a finite constant depending on $j$ and $\varphi$\,.

Given $\varepsilon>0$ we now find $\mathcal U$\,. Since the action $\Th$ is strongly continuous, for every $Z \in{\rm supp}(\varphi)$ there exists a ball $\mathbf{B}(Z,\delta_Z)$ centered in $Z$ such that if $Y \in \mathbf{B}(Z,\delta_Z)$ one has for all $|\nu|\le j$
\begin{equation}\label{beausejour}
\big\Vert\,\Th_Y(\D^\nu f)-\Th_Z(\D^\nu f)\,\big\Vert_{\B_\Th(\Si)}\,\le\,\frac{\varepsilon}{2\,C(j,\varphi)}\,.
\end{equation}
The balls $\mathbf{B}(Z,\delta_Z)$  form a covering of the compact set ${\rm supp}(\varphi)$\,, from which we extract a finite subcovering indexed by $\{Z_i\mid i\in I\}$\,. Since $\Th_{Z_i}(\D^\nu f) \in \C_0(\Si)^F$ for every $i,\nu\,,$ there exists $\mathcal U_i^\nu\in\mathcal N_F$ such that
\begin{equation}\label{vidal}
\big\vert\left[\Th_{Z_i}(\D^\nu\!f)\right]\!(\si)\,\big\vert\le\,\frac{\varepsilon}{2\,C(j,\varphi)}\,,\quad \forall\,\si \in \mathcal U_i^\nu\,.
\end{equation}
Setting 
$$
\mathcal U:=\bigcap\,\left\{\,\mathcal U_i^\nu\,\big\vert\, i\in I,|\nu|\le j\,\right\}\in\mathcal N_F
$$ 
one gets from (\ref{beausejour}) and (\ref{vidal}) 
$$
\big\vert\left[\Th_{Y}(\D^\nu\!f)\right]\!(\si)\,\big\vert\le\,\frac{\varepsilon}{C(j,\varphi)}\,,\quad \forall\,\si \in \mathcal U\ ,\ \,\forall\,Y\in{\rm supp}(\varphi)\,,\ \,\forall\,|\nu|\le j\,.
$$
Inserting this into (\ref{urias}) finishes the proof.
\end{proof}

Now we prove an estimation as (\ref{valdivia}), but with the pointwise product $\cdot$ replaced by the deformed product $\#$\,.

\begin{Proposition}\label{performardo}
For any $f\in\CC_0^\infty(\Si)^F=\C_0^\infty(\Si)^F$\,, $\varphi\in C^\infty_{{\rm c}}(\Xi)^{\rm n}_+$\,, $\varepsilon>0$ and $j \in \mathbb{N}$\,, there exists $\mathcal V\in\mathcal N_F$ such that 
\begin{equation}\label{pinto}
\big\Vert\,\chi_\mathcal V^\varphi\# f\,\big\Vert_{\B_\Th(\Si)}^{(j)}\le \varepsilon\,.
\end{equation}
\end{Proposition}

\begin{proof}
For the composition $\chi_\mathcal V^\varphi\# f$ we are going to use the representation (\ref{joseph}).

\medskip
For $G\in BC^\infty(\Xi\times\Xi;\B_\Th^\infty(\Si))$ and $j,m\in\mathbb N$ we set
\begin{equation}\label{chavela}
\begin{aligned}
\p\!G\!\p_{\B_\Th(\Si)}^{(j,m)}\,:=\,
&\max_{i\le j}\!\sum_{|(\mu,\nu)|\le m}\frac{1}{\mu!\nu!}\sup_{Y,Z\in\Xi}\big\Vert\,(\partial_Y^\mu\partial_Z^\nu G)(Y,Z)\,\big\Vert_{\B_\Th(\Si)}^{(i)}\\
=&\,\max_{i\le j}\!\sum_{|(\mu,\nu)|\le m}\frac{1}{\mu!\nu!}\sup_{Y,Z\in\Xi}\sum_{|\alpha|\le i}\frac{1}{\alpha!}\big\Vert\,\D^\alpha\!\left[(\partial_Y^\mu\partial_Z^\nu G)(Y,Z)\right]\big\Vert_{\B_\Th(\Si)}\,.
\end{aligned}
\end{equation}
By \cite[Prop. 1.6]{Rie1}, for every $k>2n$ we have estimates given by
\begin{equation}\label{jose}
\Big\Vert \int_\Xi\!\int_\Xi\!dYdZ\, e^{2i[\![Y,Z]\!]}\psi_P(Y)\psi_Q(Z)G(Y,Z)\,\Big\Vert_{\B_\Th(\Si)}^{(j)}\le C_{PQ}(k)\!\p\!G\!\p_{\B_\Th(\Si)}^{(j,2k)}\,, 
\end{equation}
where $\sum_{P,Q} C_{PQ}(k) < \infty$\,. Applying this to $G^\varphi_\mathcal V(Y,Z):=\Theta_{Y}\!\left[\chi_\mathcal V^\varphi\,\Theta_{Z-Y}(f)\right]$ 
and relying on the representation (\ref{joseph}), one gets
$$
\Big\Vert \int_\Xi\!\int_\Xi\!dYdZ\,e^{2i[\![Y,Z]\!]}\psi_P(Y)\psi_Q(Z)\Theta_{Y}\!\left[\chi_\mathcal V^\varphi\,\Theta_{Z-Y}(f)\right] \Big\Vert_{\B_\Th(\Si)}^{(j)}
\,\le C_{PQ}(k)\big\Vert\,G^\varphi_\mathcal V\,\big\Vert_{\B_\Th(\Si)}^{(j,2k)}\,. 
$$
A direct computation shows that the quantity $\,\big\Vert\,G^\varphi_\mathcal V\,\big\Vert_{\B_\Th(\Si)}^{(j,2k)}\,$ is bounded uniformly in $\mathcal V$\,, because it is dominated by a finite linear combination of terms of the form 
$$
\big\Vert\,\chi_\mathcal V^{\partial^\beta\varphi}\,\big\Vert_\infty\,\big\Vert\,\D^\gamma f\,\big\Vert_\infty\le\big\Vert\,\partial^\beta\varphi\,\big\Vert_{L^1(\Xi)}\,\big\Vert\,\D^\gamma f\,\big\Vert_\infty\,.
$$
Thus, for any $\varepsilon>0$\,, there exists $m_j\in\mathbb N$ such that for every $\mathcal V\in\mathcal N_F$ one has
$$ 
\sum_{|P|+|Q| > m_j} \Big\Vert \int_\Xi\!\int_\Xi \!dYdZ\,e^{2i[\![Y,Z]\!]}\psi_P(Y) \psi_Q(Z)\Theta_{Y}\!\left[\chi_\mathcal V^\varphi\,\Theta_{Z-Y}(f)\right] \Big\Vert_{\B_\Th(\Si)}^{(j)} \le \varepsilon/2\,.
$$
We still have to bound by $\varepsilon/2$ the remaining finite family of terms, this time for some special neighborhood $\V$ of $F$\,. Using the continuity of the action $\Th$ and the compacity of the support of $\psi_P, \psi_Q$\,, there exists a finite family 
of balls $\{\,\mathbf B(Y_i,\delta_i)\times\mathbf B(Z_i,\delta_i')\}_{i\in I}$ which covers the suport of $\psi_P\otimes\psi_Q$\,, such that for $(Y,Z)\in\mathbf B(Y_i,\delta_i)\times\mathbf B(Z_i,\delta_i')$ one has
\begin{equation}\label{siegfrid}
\big\Vert\,\chi_\mathcal V^\varphi\left(f_{Z_i-Y_i}-f_{Z-Y}\right)\big\Vert_{\B_\Th(\Si)}^{(j)} \,\le \varepsilon/2M\,,\quad \ \forall\,j \le k\,,
\end{equation}
where $f_{X}:=\Th_{-X}(f)\in\C_0^\infty(\Si)^F$ and $M$ is some positive number. In addition, by Proposition \ref{effort}, for every $i\in I$ there is some $\mathcal V_i\in\mathcal N_F$ such that 
\begin{equation}\label{etzel}
\big\Vert\,\chi^\varphi_{\mathcal V_i} f_{Z_i-Y_i}\,\big\Vert_{\B_\Th(\Si)}^{(j)} \,\le \varepsilon/2M \,,\ \quad\forall\,j \le k\,.
\end{equation}
One takes the finite intersection $\mathcal V=\bigcap_{i\in I}\!\mathcal V_i$ and then, by (\ref{siegfrid}), (\ref{etzel}) and the fact that the action $\Th$ is isometric with respect to all the semi-norms, we can estimate the compactly supported integral
$$
\begin{aligned}
\Big\Vert\int_\Xi\!\int_\Xi&\!dYdZ\,e^{2i[\![Y,Z]\!]}\psi_P(Y) \psi_Q(Z)\Theta_{Y}\!\left[\chi_\mathcal V^\varphi\,\Theta_{Z-Y}(f)\right] \Big\Vert_{\B_\Th(\Si)}^{(j)}\\
&\le\,M_{P,Q}\,\sup\left\{\big\Vert\,\chi_\mathcal V^\varphi f_{Z-Y}\,\big\Vert_{\B_\Th(\Si)}^{(j)}\,\big\vert\, Y\in{\rm supp}(\psi_P),Z\in{\rm supp}(\psi_Q)\right\}\\
&\le M_{P,Q}\,\sup_{i\in I}\big\Vert\,\chi_\mathcal V^\varphi f_{Z_i-Y_i}\,\big\Vert_{\B_\Th(\Si)}^{(j)}\\ 
&+ M_{P,Q}\sup_{i\in I}\sup\left\{\big\Vert\,\chi_\mathcal V^\varphi( f_{Z_i-Y_i}-f_{Z-Y})\big\Vert_{\B_\Th(\Si)}^{(j)}\,\big\vert\, Y\in\mathbf B(Y_i,\delta_i),Z\in\mathbf B(Z_i,\delta'_i)\right\}\\
&\le M_{P,Q}\left(\frac{\varepsilon}{2M}+\frac{\varepsilon}{2M}\right)=\frac{M_{P,Q}}{M}\,\varepsilon\,.
\end{aligned}
$$
Then,  choosing $M:=\!\!\underset{|P|+|Q| \le m_j}{\sum}\!\!C_{P,Q}M_{P,Q}$\,, one gets the estimation.
\end{proof}

Now we change the semi-norms.

\begin{Proposition}\label{pdeformado}
For any  $f\in\mathfrak C_0^\infty(\Si)^F$\,, $\varphi\in C^\infty_{{\rm c}}(\Xi)^{\rm n}_+$\,, $\varepsilon>0$ and $k \in \mathbb{N}$\,, there exists $\mathcal W\in\mathcal N_F$ such that 
\begin{equation}\label{pinto}
\big\Vert\,\chi_\mathcal W^\varphi\# f\,\big\Vert_{\mathfrak B_\Th(\Si)}^{(k)}\le\varepsilon\,.
\end{equation}
\end{Proposition}

\begin{proof}
This follows from our Proposion \ref{performardo} and the equivalence \cite[Ch. 7]{Rie1} of the families of seminorms (\ref{medel}) and (\ref{orellana}), which is a rather deep result.
\end{proof}

\noindent
{\bf End of the proof of Theorem \ref{grigore}}.
To finish the proof of Theorem \ref{grigore}, one uses Lemma \ref{sigfrid} with $\AA:=\CC_0(\Si)$ and $\JJ:=\CC_0(\Si)^F$\,. Notice the identification 
$\sp_{\CC_0(\Si)^F}(h)=\sp\!\left(h^F\big\vert\CC(F)\right)$\,. This allows us to take $f=\rho(h)\in\CC^\infty_0(\Si)^F$ in Proposition \ref{pdeformado} and to get
\begin{equation}\label{pinto}
\big\Vert\,\chi_\mathcal W^\varphi\# \rho(h)\,\big\Vert_{\mathfrak B_\Th(\Si)}^{(k)}\,\le\varepsilon
\end{equation}
under the stated conditions. The case $k=0$ is enough for our purposes.

\section{Localization and non-propagation for pseudodifferential operators}\label{onces}

We start with the simplest situation. We take $\Si$ to be a locally compact space containing $\Xi=\mathbb R^{2n}$ densely. If $\Si$ is even compact, it will be a compactification of $\Xi$\,. One denotes by $\A_\Si(\Xi)$ the $C^*$-algebra composed of restrictions to $\Xi$ of all the elements of $\mathcal C_0(\Si)$\,.
Then $\A_\Si(\Xi)$ is a $C^*$-subalgebra of ${\sf BC}(\Xi)$ which is canonically isomorphic to $\mathcal C_0(\Si)$ by the extension/restriction isomorphism. Thus the Gelfand spectrum of $\A_\Si(\Xi)$ is homeomorphic to $\Si$\,. 

Let us also assume that $\A_\Si(\Xi)$ is contained in ${\sf BC}_{{\rm u}}(\Xi)$ and it is invariant under translations. It follows easily that the action of $\Xi$ on itself by translations extends to a continuous action $\Th$ of $\Xi$ by homeomorphisms of $\Si$\,. This action is topologically transitive: $\Xi$ is an open dense orbit. Let us set $\Si_\infty:=\Si\setminus\Xi$ for the boundary.

Since $\left(\A_\Si(\Xi),\Th,\Xi\right)$ is a (commutative) $C^*$-dynamical system, one can perform Rieffel's procedure to turn it in the (non-commutative) $C^*$-dynamical system $\left(\AA_\Si(\Xi),\Th,\Xi\right)$\,. The common set of smooth vectors $\AA^\infty_\Si(\Xi)=\A^\infty_\Si(\Xi)$ is contained in $\BC^\infty(\Xi)$\,. 

It is known that $\BC^\infty(\Xi)$ is the family of smooth vectors of the $\Xi$-algebra $\BC_{{\rm u}}(\Xi)$\,, whose Rieffel quantization will be denoted by $\mathfrak B\mathfrak C_{{\rm u}}(\Xi)$\,. But on $\BC^\infty(\Xi)$\,, by the Calder\'on-Vaillancourt Theorem \cite{Fo}, one can apply the Schr\"odinger representation in $\H:=\ltwo(\X)$
\begin{equation}\label{opa}
\Op:\BC^\infty(\Xi)\to\mathbb B(\H)
\end{equation}
given in the sense of oscillatory integrals by 
\begin{equation}\label{teisn}
\big[\Op(f)u\big](x)=(2\pi)^{-n}\!\!\int_\X\!\!d y\!\int_{\X^*}\!\!\!\!d\xi\,e^{i(x-y)\cdot\xi}f\left(\frac{x+y}2,\xi\right)u(y)\,.
\end{equation}
In particular, this works for $f\in\AA^\infty_\Si(\Xi)$\,.

We also fix a closed $\Th$-invariant subset $F$ of $\,\Si_\infty$\,; it can be a quasiorbit for instance. As in Section \ref{ones}, we also consider a neighborhood basis $\mathcal N_F$ of the set $F$ in $\Si$\,. For every $\W\in\mathcal N_F$ we set $W:=\W\cap\Xi$\,. Then the function $\chi_W^\varphi$ is the restriction of $\chi_\W^\varphi$ to $\Xi$ and it belongs to $\BC^\infty(\Xi)$\,, hence $\Op\!\left(\chi_W^\varphi\right)$ makes sense as a bounded operator in $L^2(\X)$\,. 

Finally let $h\in\C_0^\infty(\Si)=\CC_0^\infty(\Si)$ be a real function and set $H:=\Op(h)=H^*$, a bounded operator in $\H:= L^2(\X)$ (we use the same notation $h$ for the restriction of $h:\Si\rightarrow\mathbb R$ to $\Xi$)\,. Relying on \cite{Ma2}, we give an operator interpretation for the set ${\rm sp}(h^F)$\,, the spectrum of $h^F:=h|_F$ computed in the non-commutative $C^*$-algebra $\CC(F)$\,. Let us write $\mathfrak Q(F)$ for the set of all quasi-orbits of the closed invariant set $F$ and denote by $\mathfrak Q_0(F)$ a subset of $\mathfrak Q(F)$ such that $F=\bigcup_{Q\in\mathfrak Q_0(F)}Q$\,. In each quasi-orbit $Q$ pick a point $\si_Q$ such that the orbit of this point is dense in $Q$\,. Then 
\begin{equation}\label{dragon}
h^{\si_Q}:\Xi\rightarrow\mathbb R\,,\quad\ h^{\si_Q}(X):=h\!\left[\Th_X(\si_Q)\right]
\end{equation}
is an element of $\BC^\infty(\Xi)$\,, to which one can apply $\Op$\,; let us set $H^{\si_Q}:=\Op\left(h^{\si_Q}\right)$\,. It can be shown \cite{Ma2} that:
\begin{itemize}
\item
The spectrum $S^Q$ of the bounded self-adjoint operator $H^{\si_Q}$ depends only of the quasi-orbit $Q$ and not of the generating point $\si_Q$\,.
\item
One has 
\begin{equation}\label{fantoma}
S^F:={\rm sp}(h^F)=\overline{\bigcup_{Q\in\mathfrak Q(F)}S^Q}=\overline{\bigcup_{Q\in\mathfrak Q_0(F)}S^Q}\,.
\end{equation} 
Of course, if $F$ is itself a quasi-orbit one can take $\mathfrak Q_0(F)=\{F\}$ and the statements simplify a lot.
\item
The set ${\rm sp}(h^F)$ is contained in the essential spectrum ${\rm sp}_{{\rm ess}}(H)$ of the initial operator $H$\,.
\item
Actually, if we cover $\Gamma$ by closed $\Th$-invariant sets $F$\,, one has
\begin{equation}\label{moroi}
{\rm sp}_{{\rm ess}}(H)=\overline{\bigcup_{F}S^F}=\overline{\bigcup_{Q\in\mathfrak Q(\Gamma)}S^Q}\,.
\end{equation}
\end{itemize}

Now we can state and prove

\begin{Theorem}\label{parriel}
Let $h\in\C_0^\infty(\Si)=\CC_0^\infty(\Si)$ be a real function and set $H:=\Op(h)$\,. Let $\rho:\mathbb R\rightarrow\mathbb R_+$ be a bounded continuous function such that $\,\supp(\rho)\cap S^F=\emptyset$\,. For every $\varepsilon>0$ there exists $\W\in \mathcal N_F$ and $\varphi\in C^\infty_{{\rm c}}(\Xi)$ such that
\begin{equation}\label{localiz}
\big\Vert\,\Op\!\left(\chi_W^\varphi\right)\rho(H)\,\big\Vert_{\mathbb B(\H)}\le\varepsilon\,.
\end{equation}
In particular, one has uniformly in $t\in\mathbb R$ and  $u\in\H$
\begin{equation}\label{nopropag}
\big\Vert\,\Op\!\left(\chi_{W}^\varphi\right)e^{itH}\rho(H)u\,\big\Vert\,\le\varepsilon\!\parallel\!u\!\parallel.
\end{equation}
\end{Theorem}

\begin{proof}
It is known (cf. \cite[Lemma 3.1]{BM} or \cite[Prop. 2.6]{Ma2}) that the mapping $\Op$ extends to a faithful (therefore isometric) representation of $\mathfrak B\mathfrak C_{{\rm u}}(\Xi)$ in $\H$\,. We can use its restriction to our algebra $\AA_\Si(\Xi)$ and apply it to the element $\rho(h)$\,. Note however that $\chi_W^\varphi$\,, element of ${\sf BC}^\infty(\Xi)\subset\mathfrak B\mathfrak C_{{\rm u}}(\Xi)$\,, has a priori no reason to belong to $\AA_\Si(\Xi)$\,.

For the first estimate we use the fact that, being a representation, $\Op$ is multiplicative and commutes with the functional calculus:
$$
\Op\!\left(\chi_W^\varphi\right)\rho(H)=\Op\!\left(\chi_W^\varphi\right)\rho[\Op(h)]=\Op\!\left(\chi_W^\varphi\right)\Op[\rho(H)]=\Op\!\left[\chi_W^\varphi\sharp\,\rho(h)\right].
$$
We denoted by $\sharp$ the Weyl composition law of symbols \cite{Fo}, corresponding isomorphically to the composition $\#$\,.
Then we use Theorem \ref{grigore}, the isomorphisms $\AA_\Si(\Xi)\cong\mathfrak C_0(\Si)$ and $\mathfrak B\mathfrak C_{{\rm u}}(\Xi)\cong\mathfrak B_\Th(\Si)$ 
and the fact that $\Op$ is an isometry to write
\begin{equation}
\big\Vert\,\Op\!\left(\chi_W^\varphi\right)\rho(H)\,\big\Vert_{\mathbb B(\H)}=\,\big\Vert\,\Op\!\left[\chi_W^\varphi\sharp\,\rho(h)\right]\big\Vert_{\mathbb B(\H)}=\,\big\Vert\,\chi_\mathcal W^\varphi\#\rho(h)\,\big\Vert_{\mathfrak B_{\Th}(\Si)}\le\varepsilon\,.
\end{equation}

As it has been said repeatedly, the second estimate (\ref{localiz}) follows from (\ref{nopropag}).
\end{proof}

A variant of Theorem \ref{parriel} involving localization along ultrafilters can be obtained in the setting of \cite{Ma3}.
{\it The Weyl system} $\pi:\Xi\rightarrow\mathbb U(\H)$ is defined for all $X\in\Xi$ and $u\in\H:= L^2(\X)$ by
\begin{equation}\label{weyl}
[\pi(X)u](y):=e^{i(y-x/2)\cdot\xi}u(y-x)\,.
\end{equation}
It is a projective unitary representation with multiplier  given in terms of the symplectic form:
\begin{equation}\label{proj}
\pi(X)\pi(Y)=\exp\left(\frac{i}{2}[\![X,Y]\!]\right)\pi(X+Y)\,,\ \quad\forall\,X,Y\in\Xi\,.
\end{equation}
We denote by 
\begin{equation}\label{auto}
\Pi:\Xi\to\mathbb \Aut[\mathbb B(\H)]\,,\ \quad\Pi(X)T:=\pi(X)T\pi(-X)
\end{equation}
the automorphism group associated to $\pi$\,. The $C^0$-vectors of this automorphism group form a $C^*$-subalgebra
\begin{equation}\label{contin}
\mathbb B^0(\H):=\{\,S\in\mathbb B(\H)\mid\, \Xi\ni X\mapsto\Pi(X)S\in\mathbb B(\H)\ \,\p\!\cdot\!\p-{\rm continuous}\,\}\,,
\end{equation}
while the $C^\infty$-vectors
\begin{equation}\label{suav}
\mathbb B^\infty(\H):=\{\,S\in\mathbb B(\H)\mid\, \Xi\ni X\mapsto\Pi(X)S\in\mathbb B(\H)\ \,\ {\rm is\ C^\infty\ in\ norm}\,\}
\end{equation}
form a dense $^*$-subalgebra.

We also denote by $\delta(\Xi)$ the family of all ultrafilters on $\Xi$ that are finer than the Fr\'echet filter. Recall from \cite{Ma3} that the essential spectrum of any self-adjoint operator $H$ belonging to $\mathbb B^0(\Xi)$ is given by
\begin{equation}\label{HWZ}
{\rm sp}_{\rm ess}(H)=\overline{\cap_{\mathcal X\in\Xi}\,{\rm sp}(H_\mathcal X)}\,,
\end{equation}
where the limits $H_\mathcal X:=\lim_{X\to\mathcal X}\Pi(X)H$ are shown to exist in the strong sense. 

\begin{Theorem}\label{conjecture}
Let $H$ be a self-adjoint operator in $\H=L^2(\X)$ belonging to $\mathbb B^\infty(\H)$\,. Let us fix an ultrafilter $\mathcal X$ on $\Xi$ finer than the Fr\'echet filter and choose a bounded continuous function $\rho:\mathbb R\to\mathbb R_+$ such that $\,{\rm supp}(\rho)\cap{\sp(H_\mathcal X)}=\emptyset$\,. 
Then for every $\epsilon>0$ and $\varphi\in C^\infty_{\rm c}(\Xi)_+^{\rm n}\,$ there exists $W\in\mathcal X$ such that 
\begin{equation}\label{estim}
\big\Vert\,\Op(\chi^\varphi_W)\,\rho(H)\,\big\Vert_{\mathbb B(\H)}\le \epsilon\,.
\end{equation}
\end{Theorem}

\begin{proof}
In \cite[Prop. 3.1]{Ma3} it has been shown that $\,\mathbb B^0(\H)=\Op\left[\mathfrak{BC}_{\rm u}(\Xi)\right]$\,, where $\mathfrak{BC}_{\rm u}(\Xi)$ is the Rieffel quantization of the $C^*$-algebra ${\sf BC}_{\rm u}(\Xi)$ of all bounded uniformly continuous functions on $\Xi$\,. As we said above, this one is the largest one on which $\Xi$ acts continuously by translations (that were denoted by $\mathcal T$). It is well-known that 
\begin{equation*}\label{legatura}
\pi(X)\Op(f)\pi(-X)=\Op[\mathcal T_X(f)]\,,\quad\ \forall\,X\in\Xi\,.
\end{equation*}
Then, clearly, one also has 
$$
\mathbb B^\infty(\H)=\Op\left[\mathfrak{BC}^\infty(\Xi)\right]=\Op\left[{\sf BC}^\infty(\Xi)\right]\,.
$$ 
Then the methods of the previous sections became available and the proof is very similar to the proof of Theorem \ref{parriel}. 
\end{proof}

Finally we treat a more general case. The action $\Th$ will no longer be composed of translations.
Our framework  starts with a continuous action $\Theta$ of $\Xi$ by homeomorphims of the locally compact space $\Sigma$\,. 
The action $\Theta$ of $\Xi$ on $\Sigma$ induces a continuous action of $\Xi$ on $\C_0(\Sigma)$ given by $\Theta_X(f)=f\circ \Theta_X$\,.
We want to realize the algebra  $(\C_0(\Sigma),\Theta,\Xi)$ in a subalgebra of $({\sf BC}_{\rm u}(\Xi),\mathfrak{T},\Xi)$ by a $\Xi$-monomorphim. For this purpose, it is convenient to have a closer look at the quasi-orbit structure of the dynamical system $(\Si,\Theta,\Xi)$ in connection with $C^*$-algebras and Hilbert space representations.

Let us use the convenient notation $\Th_\si(X):=\Th_X(\si)$\,. For each $\si\in\Si$, we denote by $E_\si:=\overline{\Theta_\si(\Xi)}$ {\it the quasi-orbit generated by} $\si$ and set
$$
\P_\si:\C_0(\Sigma)\to\BC_{\rm u}(\Xi),\quad\P_\sigma(f):=f\circ\Theta_\sigma\,.
$$
The range of the $\Xi$-morphism $\P_\si$ is called $\B_\sigma$ and it is a $\Xi$-algebra. Defining analogously $\P'_\si:\C_0(E_\sigma)\to\BC_{\rm u}(\Xi)$ one gets a $\Xi$-monomorphism with the same range $\B_\sigma$, which shows that the Gelfand spectrum of $\B_\si$ can be identified with the quasi-orbit $E_\si$\,.

For each quasi-orbit $E$, one has the natural restriction map
$$
\R_E:\C_0(\Si)\to\C_0(E),\quad\R_E(f):=f|_E,
$$
which is a $\Xi$-epimorphism. Actually one has $\P_\si=\P'_\si\circ \R_{E_\si}$\,.

Being respectively invariant under the actions $\Theta$ and $\mathcal T$, the $C^*$-algebras $\C_0(E)$
and $\B_\sigma$ are also subject to Rieffel deformation. By quantization, one gets $C^*$-algebras and morphisms
$$
\RR_E:\CC_0(\Si)\to\CC_0(E),\qquad\PP_\si:\CC_0(\Si)\to\BB_\si,\qquad\PP'_\si:\CC_0(E_\si)\to\BB_\si\,,
$$
satisfying $\PP_\si=\PP'_\si\circ \RR_{E_\si}$. While $\RR_E$ and $\PP_\si$ are epimorphisms, $\PP'_\si$ is an isomorphism.

We also need Hilbert space representations. 
For each $\Xi$-algebra $\B$, we restrict $\Op$ from $\BC^\infty(\Xi)$ to $\,\B^\infty=\BB^\infty$ (the dense
$^*$-algebra of smooth vectors of $\B\,$) and then we extend it to a faithful representation in $\H=L^2(\X)$ of the $C^*$-algebra $\BB$\,.
We can apply the construction to the $C^*$-algebras $\BB_\si$. By composing, we get a family $\big\{\,\Op_\sigma:=\Op\circ\PP_\sigma\,\vert\, \sigma\in\Sigma\,\big\}$ of representations of $\CC_0(\Sigma)$ in $\H$, indexed by the points of $\Sigma$\,. For $f\in\CC_0^\infty(\Sigma)$ one has $\PP_\sigma(f)\in\BB_\sigma^\infty=\B_\si^\infty$, and the action on $\H$ is given by
\begin{equation}\label{tein}
\big[\Op_\sigma(f)u\big](x)=(2\pi)^{-n}\!\int_\X\!d y\int_{\X^*}\!\!\!d\xi\,e^{i(x-y)\cdot\xi}f\Big[\Theta_{\left(\frac{x+y}2,\xi\right)}(\sigma)\Big]u(y)
\end{equation}
in the sense of oscillatory integrals. If the function $f$ is real, all the operators $\Op_\si(f)$ will be self-adjoint. To conclude, a single element $\,f\in\CC_0(\Sigma)\,$ leads to a family $\big\{\,H_\sigma:=\Op_\sigma(f)\,\vert\,\sigma\in\Si\,\big\}$ of bounded operators in $L^2(\X)$\,. Note that, seen as a quantization of the symbol $f$\,, (\ref{tein}) can be quite different from a Weyl operator.

\begin{Remark}\label{vigronia}
{\rm Notice that $\Op_\si$ is faithful exactly when $\PP_\si$ is injective, i.e. when $\P_\si$ is injective, which is obviously equivalent to $E_\si=\Si$\,.
Consequently, if the dynamical system is not {\it topologically transitive} (i.e. no orbit is dense in $\Si$), none of the Schr\"odinger-type representations $\Op_\si$ will be faithful. In such a case, we are not able to transform the abstract algebraic Theorem \ref{grigore} into an assertion involving operators.}
\end{Remark}

So we restrict now to the case of a topologically transitive dynamical system and study the operators $\Op_{\si}(h)$ associated to a suitable symbol $h$ and a {\it generic point} $\si\in\Si$\,, \ie a point generating a dense orbit. The non-generic points $\tau$ will define subsets $S_\tau$ of the essential spectrum of $\Op_{\si}(h)$
as well as regions of non-propagation for its evolution group.

\begin{Theorem}\label{pariell}
Let $(\Sigma,\Theta,\Xi)$ a topologically transitive dynamical system, $(\C_0(\Sigma),\Theta,\Xi)$ its associated C*-dynamical system and $\si\in\Si$ a generic point.

For a fixed real function $h\in\C_0^\infty(\Si)=\CC_0^\infty(\Si)$ set $H_\si:=\Op_{\si}(h)$\,. Choose a non-generic point $\tau\in\Si$\,, denote by $E_\tau$ its quasi-orbit (strictly contained in $\Si$) and set $H_\tau:=\Op_{\tau}(h)$\,. Let $\rho:\mathbb R\rightarrow\mathbb R_+$ be a continuous function such that $\,\supp(\rho)\cap{\rm sp}(H_\tau)=\emptyset$\,. 

For every $\varepsilon>0$ there exists a neighborhood $\W$ of $E_\tau$ in $\Si$ and a positive function $\varphi\in C^\infty_{{\rm c}}(\Xi)$ with $\int_\Xi\varphi=1$ such that
\begin{equation}\label{localiz}
\big\Vert\,\Op_{\si}\!\left(\chi_\W^\varphi\right)\rho(H_\si)\,\big\Vert_{\mathbb B(\H)}\le\varepsilon\,.
\end{equation}
In particular, one has uniformly in $t\in\mathbb R$ and  $u\in L^2(\X)$
\begin{equation}\label{nopropag}
\big\Vert\,\Op_{\si}\!\left(\chi_{\W}^\varphi\right)e^{itH_\si}\rho(H_\si)u\,\big\Vert\le\varepsilon\!\parallel\!u\!\parallel.
\end{equation}
\end{Theorem}

\begin{proof}
Since $(\Sigma,\Theta,\Xi)$ is topologically transitive and $\si$ is generic, the mapping $\Op_{\si}$ is a faithful representation of our algebra $\CC_0(\Si)$\,.

Then the present result follows easily from Theorem \ref{grigore}, along the lines of the proof of Theorem \ref{parriel}. The role of the closed invariant subset $F$ is played here by the non-generic quasi-orbit $E_\tau$\,.
\end{proof}

\begin{Remark}\label{lustr}
{\rm Recall that we set $\Op_{\si}\!\left(\chi_\W^\varphi\right):=\Op\left(\chi_\W^\varphi\circ\Th_\si\right)$\,. The function $\chi_\W^\varphi\circ\Th_\si$ is a mollified version of $\chi_\W\circ\Th_\si$\,, which in its turn is the characteristic function of the subset $\Th_\si^{-1}(\W)$ of the phase-space $\Xi$\,. By our choice of the points $\si,\tau$\,, one has $\mathcal O_\tau\cap\mathcal O_\si=\emptyset$ and $E_\tau\subset\mathcal W\subset E_\si=\Si$\,, where the inclusions are strict.
}
\end{Remark}

\begin{Remark}\label{restr}
{\rm For a better understanding of the dependence on the points $\si$ and $\tau$\,, we recall some results from \cite{Ma2}. If $\si_1,\si_2$ belong to the same orbit ($\mathcal O_{\si_1}=\mathcal O_{\si_2}$)\,, the two operators $H_{\si_1}$ and $H_{\si_2}$ are unitarily equivalent. If the two points only generate the same quasi-orbit 
$E_{\si_1}:=\overline{\mathcal O_{\si_1}}=\overline{\mathcal O_{\si_1}}=:E_{\si_2}$ 
they may not be unitarily equivalent, but they still have the same spectrum and the same essential spectrum. In applications, very often, there is a privileged generic point $\si_0$ defining an interesting Hamiltonian $H_{\si_0}$ as in (\ref{tein}) and the remaining objects are auxiliary constructions. Their usefulness comes from the fact that the behavior of the symbol requires a topological dynamical system encoding spectral information.
  }
\end{Remark}

\medskip
\noindent
{\bf Acknowledgements.}
M. M\u antoiu has been supported by the Fondecyt Project 1120300. The authors are grateful to Serge Richard for a critical reading of the manuscript.

\nocite{*}
\bibliographystyle{cdraifplain}
\bibliography{xampl}

\end{document}